\documentclass{amsart}

\usepackage{amssymb}

\usepackage[theoremfont]{newtxtext}
\usepackage[varvw,upint,smallerops]{newtxmath}

\usepackage[papersize={17cm,24cm},text={12.5cm,19.5cm},centering]{geometry}

\usepackage[initials,lite]{amsrefs}

\usepackage{tikz-cd}
\tikzcdset{row sep/normal=3.5em}

\usepackage{verbatim}

\usepackage{amsfonts}
\usepackage[shortlabels]{enumitem}
\setlist{noitemsep}

\usepackage{hyperref}

\numberwithin{equation}{section}

\theoremstyle{definition}

\newtheorem{theorem}{Theorem}[section]
\newtheorem*{theorem41}{Theorem \ref{main}}
\newtheorem{proposition}{Proposition}[section]
\newtheorem{corollary}{Corollary}[section]
\newtheorem{lemma}{Lemma}[section]
\newtheorem{definition}{Definition}[section]
\newtheorem*{example*}{Example}

\newcommand{\CA}{\mathcal{A}}
\newcommand{\op}{\textup{op}}
\DeclareMathOperator{\im}{im}
\newcommand{\Z}{\mathbb{Z}}

\newcommand{\T}{\mathsf{T}}

\renewcommand{\nplus}{[n{+}1]}
\newcommand{\nminus}{[n{-}1]}

\renewcommand{\]}{]\!]}

\begin{document}

\title{The Dold--Kan theorem for paracyclic modules}

\author{Ezra Getzler}

\address{Department of Mathematics, Northwestern University}

\email{getzler@northwestern.edu}

% \dedicatory{This article is dedicated to Daniel Kan and Daniel
% Quillen}

\begin{abstract}
  We study the Karoubi operator on the unnormalized chain complex of a
  paracyclic module; its restriction to the normalized chain complex
  has previously been considered by Dwyer and Kan, and in the cyclic
  case by Cuntz and Quillen. We obtain a direct proof of the Dold-Kan
  theorem for paracyclic modules of Dwyer and Kan, by relating the
  Karoubi operator to projection to the normalized subcomplex.
\end{abstract}

% \subjclass{18G31; 16E40}
%  paracyclic modules, cyclic modules, Dold--Kan theorem;

\maketitle

\section{Paracyclic modules}

Let $\Lambda_\infty$ be the category whose objects are the natural numbers
$[n]$, $n\ge0$, and whose morphisms are weakly monotone periodic
functions:
\begin{multline*}
  \Lambda_\infty([m],[n]) \\
  = \{ f:\Z\to\Z \mid \text{$f(j)\le f(k)$ if $j\le k$, and
    $f(j+m+1)=f(j)+n+1$} \} .
\end{multline*}
This category was denoted $L$ by Dwyer and Kan \cite{DK}; the notation
$\Lambda_\infty$ is due to Feigin and Tsygan \cite{FT}. (Lurie \cite{Lurie} uses
the notation $\Delta_\circlearrowright$.)

The category $\Lambda_\infty$ has subcategories
\begin{equation*}
  \Lambda_+([m],[n]) = \{ f\in\Lambda_\infty([m],[n]) \mid f(0)\ge0 \} ,
\end{equation*}
denoted $K$ by Dwyer and Kan \cite{DK}, and
\begin{equation*}
  \Delta([m],[n]) = \{ f\in\Lambda_+([m],[n]) \mid f(m)\le n \} .
\end{equation*}

Define the morphisms $\varepsilon_{n,i}\in\Delta(\nminus,[n])$,
$0\le i\le n$, $\eta_{n,i}\in\Delta(\nplus,[n])$, $0\le i\le n$, and
$\eta_{n,n+1}\in\Lambda_+(\nplus,[n])$, by
\begin{align*}
  \varepsilon_{n,i}(j) &=
  \begin{cases}
    j , & 0\le j<i , \\
    j+1 , & i\le j\le n-1 ,
  \end{cases} \intertext{and}
  \eta_{n,i}(j) &=
  \begin{cases}
    j , & 0\le j\le i , \\
    j-1 , & i<j\le n+1 .
  \end{cases}
\end{align*}
These morphisms satisfy the relations
\begin{align*}
  \varepsilon_{n+1,i} \varepsilon_{n,j} &= \varepsilon_{n+1,j+1} \varepsilon_{n,i} , \qquad \, \, 0\le i\le j\le n , \\
  \eta_{n,j} \varepsilon_{n+1,i} &=
                    \begin{cases}
                      \tau_n , & j-i=n+1 , \\
                      \varepsilon_{n,i} \eta_{n-1,j-1} , & 1\le j-i\le n , \\
                      1 , & j-i=0,-1, \\
                      \varepsilon_{n,i-1} \eta_{n-1,j} , & -n-1\le j-i<-1 ,
                    \end{cases} \\
  \eta_{n-1,j} \eta_{n,i} &= \eta_{n-1,i} \eta_{n,j+1} , \qquad \!\! 0\le i\le j\le n .
\end{align*}
The morphism $\tau_n=\eta_{n+1,n}\varepsilon_{n+1,0}\in\Lambda_\infty([n],[n])$ is the function
\begin{equation*}
  \tau_n(i) = i+1 .
\end{equation*}
We have $\tau_n^{n+1}\varepsilon_{n,i}=\varepsilon_{n,i}\tau_{n-1}^n$ and
$\tau_n^{n+1}\eta_{n,i}=\eta_{n,i}\tau_{n+1}^{n+2}$. It follows that
$\Lambda_+$ is generated by the morphisms
\begin{equation*}
  \{ \varepsilon_{n,0} : \nminus \to [n] \}_{n>0} \cup \{ \eta_{n,n+1} : \nplus \to
  [n]\}_{n\ge0} ,
\end{equation*}
and $\Lambda_\infty$ is obtained from $\Lambda_+$ by inverting
$\{ \tau_n^{n+1} : [n] \to [n] \}_{n\ge0}$.

Let $\ell\ge1$. The cyclic category $\Lambda_\ell$ is the quotient of
$\Lambda_\infty$ by the relations $(\tau_n^{n+1})^\ell=1$, $n\ge0$.  For
$\ell=1$, we denote $\Lambda_\ell$ by $\Lambda$.

There is an isomorphism of categories $f\mapsto f^\circ$ between
$\Lambda_\infty$ and its opposite $\Lambda^\op_\infty$, acting on generators by
interchanging $\varepsilon_{n,i}$ and $\eta_{n-1,n-i}$, and fixing
$\tau_n$. In particular, it exchanges the generators
$\varepsilon_{n,0}$ and $\eta_{n-1,n}$ of $\Lambda_\infty$. This isomorphism descends to an
isomorphism of $\Lambda_\ell$ with $\Lambda_\ell^\op$, for $\ell\ge1$.

\begin{definition}
  Let $\CA$ be a pre-additive category (a category enriched in abelian
  groups).  A \textbf{paracyclic} (respectively \textbf{duplicial},
  \textbf{cyclic}, \textbf{simplicial}) module $M_\bullet$ in $\CA$ is a
  presheaf on $\Lambda_\infty$ (respectively $\Lambda_+$,
  $\Lambda_\ell$, $\Delta$) taking values in $\CA$.
\end{definition}

Denote the (co)action of the morphisms $\varepsilon_{n,i}$ and
$\eta_{n,i}$ on a duplicial module $M_\bullet$ by
$\partial_{n,i}:M_n\to M_{n-1}$ and $s_{n,i}:M_n\to M_n$, and the (co)action of
the morphisms $\tau_n$ by $(-1)^nt_n:M_n\to M_n$. (Our sign convention for
$t_n$ is not the usual one, but it serves to simplify the signs in
several formulas.) The relations among the generators of $\Lambda_+$ become
the following relations on duplicial modules $M_\bullet$:
\begin{align*}
  \partial_{n,k} \partial_{n+1,j} &= \partial_{n,j} \partial_{n+1,k+1} , \qquad \,\, 0\le j\le k\le n , \\
  \partial_{n+1,j} s_{n,k} &=
                    \begin{cases}
                      t_n , & k-j=n+1 , \\
                      s_{n-1,k-1} \partial_{n,j} , & 1\le k-j\le n , \\
                      1 , & k-j=0,-1, \\
                      s_{n-1,k} \partial_{n,j-1} , & -n-1\le k-j<-1 ,
                    \end{cases} \\
  s_{n,j} s_{n-1,k} &= s_{n,k+1} s_{n-1,j} , \qquad \!\! 0\le j\le k\le n .
\end{align*}
These relations imply that
\begin{align*}
  \partial_{n,i} t_n &=
                \begin{cases}
                  - t_{n-1} \partial_{n,i+1} , & 0\le i< n , \\
                  (-1)^n \partial_{n,0} , & i=n ,
                \end{cases} \intertext{and}
  t_{n+1} s_{n,i} &=
                \begin{cases}
                  (-1)^{n+1} s_{n,n+1} , & i=0 , \\
                  - s_{n,i-1} t_n , & 0< i\le n+1 .
                \end{cases}
\end{align*}
Denote the action of $t_n^{n+1}$ on $M_n$ by $\T_n$. We see that
$\T_{n-1}\partial_{n,i}=\partial_{n,i}\T_n$ and $\T_{n+1}s_{n,i}=s_{n,i}\T_n$.

If $M_\bullet$ is a duplicial module, it is convenient to introduce the
operations $\delta_n=\partial_{n,0}$ and $\sigma_n=(-1)^ns_{n,n+1}$. These operations
play a primary role in what follows. Note that $t_n=\delta_{n+1}\sigma_n$.

\begin{example*}
  The cyclic nerve of Connes and Tsygan (cf.\ \cite{CQ}) is an example
  of a duplicial module. Let $\CA$ be a tensor category, that is, a
  symmetric monoidal pre-additive category. Let $A$ be an associative
  algebra in $\CA$, and let $\gamma$ be an endomorphism of $A$. Then the
  twisted bar construction $B_n(A,\gamma)=A^{\otimes n+1}$ is a duplicial
  module. (When $\gamma=1$ is the identity morphism, we write
  $B_\bullet A$ in place of $B_\bullet(A,1)$.) In the special case where
  $\CA$ is the category of $R$-modules for a commutative ring $R$, and
  $V\otimes W$ is the tensor product $V\otimes_RW$, we may write the operations
  defining this duplicial structure at the level of elements:
  \begin{align*}
    \partial_{n,i}(a_0\otimes\cdots\otimes a_n) &=
                          \begin{cases}
                            a_0\otimes\cdots\otimes a_{i-1}a_i\otimes\cdots\otimes a_n , & 0<i\le n , \\
                            a_1\otimes\cdots\otimes a_n\gamma(a_0) , & i=0 , \\
                          \end{cases} \\
    s_{n,i}(a_0\otimes\cdots\otimes a_n) &= a_0\otimes\cdots\otimes a_{i-1} \otimes 1 \otimes a_i \otimes\cdots\otimes a_n ,
                          \quad 0\le i\le n+1 , \\
    t_n(a_0\otimes\cdots\otimes a_n) &= (-1)^n a_1\otimes\cdots\otimes a_n \otimes \gamma(a_0) .
  \end{align*}
  We see that
  $\T_n(a_0\otimes\cdots\otimes a_n)=\gamma(a_0)\otimes\cdots\otimes\gamma(a_n)$. This duplicial structure is
  paracyclic if and only if $\gamma$ is invertible, and cyclic if and only
  if $\gamma$ has finite order $\ell$.
\end{example*}

An \textbf{idempotent} $(A,p)$ in a pre-additive category $\CA$ is a
pair consisting of an object $A$ of $\CA$ and $p:A\to A$ an idempotent
endomorphism of $A$, that is, $p^2=p$. If $p:A\to A$ is an idempotent,
then so is $p^\perp=1-p$, and the image of the idempotent $p$ is the
kernel of $p^\perp$ (and vice versa).

Let $M_\bullet$ be a simplicial module. Dold and Puppe \cite{DP}*{Section 3}
define idempotents
\begin{equation*}
  p_n = (1-s_{n-1,0}\partial_{n,1})\ldots(1-s_{n-1,n-1}\partial_{n,n}) : M_n \to M_n .
\end{equation*}
The kernel of the idempotent $p_n:M_n\to M_n$, if it exists, is the
module of degenerate chains $D_n(M)$. The image of $p_n$ is the module
$N_n(M)$ of normalized chains. Provided these modules exist, we have
the direct sum decomposition
\begin{equation*}
  M_n = N_n(M) \oplus D_n(M) .
\end{equation*}
The module $N_n(M)$ is the submodule on which the face maps
$\partial_{n,i}$, $1\le i\le n$, vanish:
\begin{equation*}
  N_n(M) = \bigcap_{i=1}^n \ker\bigl( \partial_{n,i}:M_n\to M_{n-1} \bigr) \subset M_n .
\end{equation*}
Similarly, the module $D_n(M)$ is the sum of the images of the
degeneracy maps $s_{n-1,i}$, $0\le i<n$:
\begin{equation*}
  D_n(M) = \sum_{i=0}^{n-1} \im\bigl( s_{n-1,i}:M_{n-1}\to M_n \bigr) \subset M_n .
\end{equation*}

In the case where $M_\bullet=B_\bullet A$ is the nerve of an algebra, the
normalized chain complex $N_\bullet(B_\bullet A)$ may be identified with the
quotient $B_\bullet A/D_\bullet(B_\bullet A)$:
\begin{equation*}
  N_n(B_\bullet A) \cong (A/R)^{\otimes_Rn}\otimes_RA .
\end{equation*}
Under the map sending an element
\begin{equation*}
  (a_0\bmod R)\otimes\cdots\otimes(a_{n-1}\bmod R)\otimes a_n
\end{equation*}
to $da_0\ldots da_{n-1}\,a_n$, we obtain a natural identification of the
normalized chain complex with the noncommutative de Rham complex of
Karoubi:
\begin{equation*}
  \Omega_nA \cong N_n(B_\bullet A) .
\end{equation*}

The following property of a pre-additive category $\CA$ was introduced
by Freyd \cite{Freyd} (cf.\ B\"uhler \cite{Buhler}*{Section 7}).
\begin{definition}
  A \textbf{section} of a morphism $f:M\to N$ is a morphism $g:N\to M$
  with
  \begin{equation*}
    fg=1_N .
  \end{equation*}
  A morphism with a section is called a \textbf{retraction}. A
  pre-additive category $\CA$ is \textbf{weakly idempotent complete}
  if every retraction has a kernel.
\end{definition}

\begin{lemma}
  If $\CA$ is weakly idempotent complete, the idempotent $p_n$ splits,
  implying the existence of the modules $N_n(M)$ and $D_n(M)$.
\end{lemma}
\begin{proof}
  Consider the filtration
  \begin{equation*}
    F_iM_n = \bigcap_{j=n-i+1}^n \ker\bigl( \partial_{n,j}:M_n\to M_{n-1} \bigr) \subset M_n .
  \end{equation*}
  Thus $F_iM_n=M_n$, and $F_nM_n=N_n(M)$.  The existence of $F_iM_n$
  is shown by induction on $i$: $F_{i+1}M_n$ is the kernel of the
  split epimorphism
  \begin{equation*}
    \partial_{n,n-i} : F_iM_n \to F_iM_{n-1} ,
  \end{equation*}
  with section $s_{n-1,n-i-1}:F_iM_{n-1} \to F_iM_n$.
\end{proof}

\begin{definition}
  Let $M_\bullet$ be a duplicial module.
  \begin{enumerate}[\alph*)]
  \item The \textbf{Karoubi operator} is the endomorphism of $M_n$
    \begin{equation*}
      \kappa_n = \delta_{n+1} \sigma_n + \sigma_{n-1}\delta_n .
    \end{equation*}
  \item The \textbf{Dwyer--Kan operator} is the endomorphism of $M_n$
    \begin{equation*}
      \pi_n = \delta_{n+1} \kappa_{n+1}^n \sigma_n .
    \end{equation*}
  \end{enumerate}
\end{definition}

In the special case $M_\bullet=B_\bullet A$, the restriction of the Karoubi
operator to the normalized complex was introduced by Cuntz and Quillen
\cite{CQ}*{Section 1}. They show that the restriction of the Dwyer--Kan
operator to the normalized complex equals the identity
\cite{CQ}*{Section 1, Eq.\ (19)}. The restriction of the Dwyer--Kan operator
to the normalized chain complex of a duplicial module was considered
by Dwyer and Kan \cite{DK}*{Section 6}.

Let $\CA$ be a weakly idempotent additive category.  The purpose of
this article is to give a proof of the following result.
\begin{theorem41}
  $\pi_n= p_n \circ \T_n = \T_n \circ p_n$
\end{theorem41}
In particular, for cyclic modules with $\ell=1$, the Dwyer--Kan operator
$\pi_n$ equals the Dold--Puppe idempotent $p_n$, giving a new proof of
\cite{DK}*{Corollary~6.6}.

In the last section of this article, we use the results of Dwyer and
Kan to show that every simplicial module has a canonical cyclic
structure, showing that Theorem~\ref{main} may be viewed as a
generalization of the Dold--Kan theorem.

\section{Normalization of simplicial modules}
\label{normalization}

For $0\le i\le n$, consider the operators
\begin{equation*}
  p_{n,i} = (1-s_{n-1,i}\partial_{n,i+1})\ldots(1-s_{n-1,n-1}\partial_{n,n}) : M_n \to M_n .
\end{equation*}
In particular $p_n=p_{n,0}$. We have
\begin{equation*}
  \partial_{n,k}(1-s_{n-1,j}\partial_{n,j+1}) =
  \begin{cases}
    (1-s_{n-2,j-1}\partial_{n-1,j})\partial_{n,k} , & k<j , \\
    \partial_{n,j} - \partial_{n,j+1} , & k=j , \\
    0 , & k=j+1 , \\
    (1-s_{n-2,j}\partial_{n-1,j+1})\partial_{n,k} , & k>j+1 .
  \end{cases}
\end{equation*}
It follows that $\partial_{n,k}p_{n,i}=0$ for $i<k\le n$, and hence that
$p_n$ is idempotent (as are the operators $p_{n,i}$, $i>0$).  Since
$p_n^\circ=p_n$, it also follows that
$p_ns_{n-1,k}=(\partial_{n,n-k}p_n){}^\circ=0$ for $0\le k\le n-1$.

A simplicial module $M_\bullet$ gives rise to a chain complex: it carries
the differential
\begin{equation*}
  b_n = \sum_{i=0}^n (-1)^i \partial_{n,i} : M_n \to M_{n-1} .
\end{equation*}
We have
\begin{align*}
  b_{n-1}b_n &= \sum_{0\le j\le k\le n-1} (-1)^{j+k} \, \partial_{n-1,k} \partial_{n,j}
               + \sum_{0\le k< j\le n} (-1)^{j+k} \, \partial_{n-1,k} \partial_{n,j} \\
             &= \sum_{0\le j\le k\le n-1} (-1)^{j+k} \, \partial_{n-1,j} \partial_{n,k+1}
               + \sum_{0\le k< j\le n} (-1)^{j+k} \, \partial_{n-1,k} \partial_{n,j} = 0 .
\end{align*}

\begin{proposition}
  \begin{equation*}
    b_np_{n,i} =
    \begin{cases}
      b_np_{n,1} , & i=0 \\
      (1-s_{n-2,i-1}\partial_{n-1,i})b_np_{n,i+1} , & 0<i<n
    \end{cases}
  \end{equation*}
\end{proposition}
\begin{proof}
  We calculate as follows:
  \begin{align*}
    b_np_{n,i} &= \sum_{k=0}^{i-1} (-1)^k \partial_{n,k}p_{n,i} + (-1)^i
                 \partial_{n,i}p_{n,i} \\
      &= \sum_{k=0}^{i-1} (-1)^k \partial_{n,k}(1-s_{n-1,i}\partial_{n,i+1})p_{n,i+1} \\
      &\qquad + (-1)^i \partial_{n,i}(1-s_{n-1,i}\partial_{n,i+1})p_{n,i+1} \\
      &= \sum_{k=0}^{i-1} (-1)^k (1-s_{n-2,i-1}\partial_{n-1,i})\partial_{n,k}p_{n,i+1} +
        (-1)^i ( \partial_{n,i} - \partial_{n,i+1})p_{n,i+1} .
  \end{align*}
  If $i=0$, this equals $b_np_{n,1}$, while if $i>0$, it equals
  \begin{multline*}
    \sum_{k=0}^{i+1} (-1)^k (1-s_{n-2,i-1}\partial_{n-1,i})\partial_{n,k}p_{n,i+1} \\
    \begin{aligned}
      &+ (-1)^i s_{n-2,i-1}\partial_{n-1,i}( \partial_{n,i} - \partial_{n,i+1})p_{n,i+1} \\
      &= (1-s_{n-2,i-1}\partial_{n-1,i})b_np_{n,i+1} \\
      &\quad + (-1)^i s_{n-2,i-1}\partial_{n-1,i}( \partial_{n,i} - \partial_{n,i+1})p_{n,i+1} .
    \end{aligned}
  \end{multline*}
  The last line vanishes, since $\partial_{n-1,i}\partial_{n,i} = \partial_{n-1,i}\partial_{n,i+1}$.
\end{proof}

\begin{corollary}
  $b_np_{n,i}=p_{n-1,i-1}b_n$ for $i>0$ and $b_np_n=p_{n-1}b_n$
\end{corollary}

It was proved by Eilenberg and Mac Lane \cite{EM1} that the complexes
$(M_\bullet,b)$ and $(N_\bullet(M),b)$ are homotopy equivalent. The following
expression for the homotopy is taken from Epstein
\cite{Epstein}*{Proposition 2.3}.
\begin{proposition}
  Let $M_\bullet$ be a simplicial module. The operator
  \begin{equation*}
    \varphi_n = \sum_{i=0}^n (-1)^i s_{n,i} p_{n,i} : M_n \to M_{n+1}
  \end{equation*}
  satisfies $b_{n+1}\varphi_n + \varphi_{n-1}b_n = p_n - 1$.
\end{proposition}
\begin{proof}
  For $0\le i\le n$,
  \begin{align*}
    b_{n+1} s_{n,i} p_{n,i}
    &= \sum_{k=0}^{i-1} (-1)^k \partial_{n+1,k} s_{n,i} p_{n,i}
      + (-1)^i ( \partial_{n+1,i} - \partial_{n+1,i+1} ) s_{n,i} p_{n,i} \\
    & \quad + \sum_{k=i+2}^{n+1} (-1)^k \partial_{n+1,k} s_{n,i} p_{n,i} \\
    &= \sum_{k=0}^{i-1} (-1)^k s_{n-1,i-1} \partial_{n,k} p_{n,i}
      + \sum_{k=i+2}^{n+1} (-1)^k s_{n-1,i}  \partial_{n,k-1} p_{n,i} \\
    &= \sum_{k=0}^n (-1)^k s_{n-1,i-1} \partial_{n,k} p_{n,i}
      - (-1)^i s_{n-1,i-1} \partial_{n,i} p_{n,i} \\
    &= s_{n-1,i-1} b_n p_{n,i} + (-1)^i ( p_{n,i-1} - p_{n,i} ) \\
    &= s_{n-1,i-1} p_{n-1,i-1} b_n + (-1)^i ( p_{n,i-1} - p_{n,i} ) .
  \end{align*}
  Summing over $i$, the result follows
\end{proof}

The following theorem of Dold and Puppe \cite{DP} is a reformulation
of theorems of Dold \cite{Dold} and Kan \cite{Kan}. Dold and Puppe
state this theorem for $\CA$ abelian, but their proof applies more
generally. (Recall that a pre-additive category is additive if it has
finite sums.)

For $0\le i_1<\cdots<i_k<n$, let
\begin{equation*}
  \begin{cases}
    f_{n,i_1\ldots i_k} = \partial_{n-k+1,i_1+1}\ldots \partial_{n,i_k+1} : M_n \to M_{n-k} , &
      \text{and} \\
    g_{n-k,i_1\ldots i_k} = s_{n-1,i_k}\ldots s_{n-k,i_1} : M_{n-k} \to M_n .
  \end{cases}
\end{equation*}
\begin{theorem}[Dold--Kan decomposition]
  \label{DK}
  Let $\CA$ be a weakly idempotent additive category. Let $M_\bullet$ be a
  simplicial module. The operators
  \begin{equation*}
    g_{n-k,i_1\ldots i_k} p_{n-k} f_{n,i_1\ldots i_k}
  \end{equation*}
  are commuting idempotents, whose sum
  \begin{equation*}
    \sum_{k=0}^n \sum_{0\le i_1<\cdots<i_k<n} g_{n-k,i_1\ldots i_k} p_{n-k} f_{n,i_1\ldots i_k}
  \end{equation*}
  is the identity of $M_n$.  \qed
\end{theorem}

The behavior of the coface maps $\partial_{n,i}$, $0<i\le n$, and the
codegeneracy maps $s_{n,i}$, $0\le i\le n$, under the Dold--Kan
decomposition is determined by the simplicial relations:
\begin{align*}
  \partial_{n,i} g_{n-k,i_1\ldots i_k}p_{n-k} &=
  \begin{cases}
    g_{n-k,i_1-1\ldots i_{q-1}-1 \widehat{\imath_q}i_{q+1}\ldots i_k} p_{n-k} , &
    i-i_q\in\{0,1\} , \\
    0 , & \text{otherwise} ,
  \end{cases} \intertext{and}
  s_{n,i} g_{n-k,i_1\ldots i_k}p_{n-k} &=
  g_{n-k,i_1-1\ldots i_{q-1}-1 \widehat{\imath_q}i_{q+1}\ldots i_k} p_{n-k} ,
                                  \qquad i_q\ge i>i_{q+1} .
\end{align*}
The Dold--Kan theorem reconstructs a simplicial module $M_\bullet$ from the
chain complex $(N_\bullet(M),b_\bullet)$. Written in terms of the decomposition
\begin{equation*}
  M_n \cong \sum_{k=0}^n \sum_{0\le i_k<\cdots<i_1<n} g_{n-k,i_1\ldots i_k} N_{n-k}(M)
\end{equation*}
of Theorem~\ref{DK}, the simplicial operations $\partial_{n,i}$,
$0<i\le n$, and $s_{n,i}$, $0\le i\le n$, only depend on the underlying
graded module of $N_\bullet(M)$, and not on the differential
$b_\bullet$. The formula for the remaining operation
$\delta_n=\partial_{n,0}$ depends, by contrast, on the differential $b$:
\begin{equation*}
  \delta_n g_{n-k,i_1\ldots i_k} p_{n-k}
  = \begin{cases}
    g_{n-k,i_2-1\ldots i_k-1} p_{n-k} , & i_1=0 \\
    g_{n-k-1,i_1-1\ldots i_k-1} b_{n-k} p_{n-k} , & i_1>0 .
  \end{cases}
\end{equation*}
Indeed, commuting $\partial_{n,0}$ past the codegeneracy morphisms
$s_{n-k+j-1,i_j}$, $j>1$, we apply the formula
$\partial_{n-k+1,0}s_{n-k,0}=1$ if $i_1=0$. If $i_1>0$, we obtain
\begin{equation*}
  g_{n-k-1,i_1-1\ldots i_k-1} \bigl( \partial_{n-k,0} p_{n-k} \bigr)
  = g_{n-k-1,i_1-1\ldots i_k-1} \bigl( b_{n-k,0} p_{n-k} \bigr) .
\end{equation*}
We conclude that the structure of a simplicial module is entirely
determined by the action of the differential
$b_\bullet=\delta_\bullet$ on the normalized subcomplex
$N_\bullet(M)$: this is the Dold--Kan theorem.
\begin{theorem}
  Let $\CA$ be a weakly idempotent complete additive category. The
  normalization functor $N_\bullet$ induces an equivalence of categories
  between simplicial modules $M_\bullet$ in $\CA$ and chain complexes in
  $\CA$.
\end{theorem}

\section{Duplicial modules and the Karoubi operator}

The following definition is due to Dwyer and Kan \cite{DK}.
\begin{definition}
  A \textbf{duchain complex} is a graded module $\{V_n \mid n\ge0\}$ with a
  pair of differentials $b_n:V_n\to V_{n-1}$ and $d_n:V_n\to V_{n+1}$.
\end{definition}

A duplicial module $M_\bullet$ gives rise to a duchain complex: it carries
the pair of differentials
\begin{align*}
  b_n &= \sum_{i=0}^n (-1)^i \partial_{n,i} : M_n \to M_{n-1} , &
  d_n &= \sum_{i=0}^{n+1} (-1)^i s_{n,i} : M_n \to M_{n+1} .
\end{align*}
The differential $d$ was introduced by Dwyer and Kan \cite{DK}, and
its restriction to the normalized complex generalizes Karoubi's
noncommutative de~Rham differential (cf.\ Cuntz and Quillen \cite{CQ}).
The proof that it is a differential is formally identical to the proof
that $b$ is a differential, since $b_n$ and $(-1)^n d_{n-1}$ are
exchanged by the anti-involution $(-)^\circ$. (These differentials are
denoted $d$ and $\delta$ in \cite{DK}.)

The simplicial relations and the formula for $d_n$ give a formula for
the behavior of the extra degeneracy operator
$\sigma_n=(-1)^n s_{n,n+1}$ on $M_n$ under the Dold--Kan decomposition:
\begin{align}
  \label{s}
  \sigma_n g_{n-k,i_1\dots i_k} p_{n-k}
  &= (-1)^k g_{n-k+1,i_1 \ldots i_k} \sigma_{n-k} p_{n-k} \\ \notag
  &= (-1)^{k-1} g_{n-k+1,i_1 \ldots i_k} d_{n-k} p_{n-k} \\ \notag
  &+ \sum_{\ell=0}^k \sum_{i_\ell<i<i_{\ell+1}} (-1)^{i+k-\ell} \, g_{n-k,i_1\ldots i_\ell
    ii_{\ell+1} \ldots i_k} p_{n-k} .
\end{align}
(In this formula, $i_0$ stands for $-1$ and $i_{k+1}$ stands for $n$.)
Using this formula, Dwyer and Kan \cite{DK} establish the following
theorem.
\begin{theorem}
  \label{DwK}
  The normalization functor $N_\bullet$ induces an equivalence of categories
  between duplicial modules $M_\bullet$ in $\CA$ and duchain complexes in
  $\CA$.
\end{theorem}

One reason for the importance of the Karoubi operator is that it may
be expressed solely in terms of $b$ and $d$.
\begin{proposition}
  We have
  \begin{equation*}
    \kappa_n = 1 - b_{n+1}d_n - d_{n-1}b_n
    = (1-b_{n+1}d_n)(1-d_{n-1}b_n) ,
  \end{equation*}
  and hence $b_n\kappa_n=\kappa_{n-1}b_n$ and $d_n\kappa_n=\kappa_{n+1}d_n$.
\end{proposition}
\begin{proof}
  We have
  \begin{align*}
    d_{n-1}b_n &= \sum_{i,j=0}^n (-1)^{i+j} s_{n-1,j} \partial_{n,i} \\
    &= (-1)^n s_{n-1,n}\partial_{n,0} \\
    &\quad - \sum_{k=1}^n (-1)^k \sum_{i=0}^{n-k+1} s_{n-1,i+k-1} \partial_{n,i} +
      \sum_{k=1}^n (-1)^k \sum_{i=k+1}^{n+1} s_{n-1,i-k-1} \partial_{n,i-1}
      \intertext{and}
    b_{n+1}d_n
    &= \sum_{i,j=0}^{n+1} (-1)^{i+j} \partial_{n+1,i} s_{n,j} \\
    &= (-1)^{n+1} \partial_{n+1,0} s_{n,n+1} + \sum_{k=1}^n (-1)^k
      \sum_{i=0}^{n-k+1} \partial_{n+1,i} s_{n,i+k} \\
    &\quad + \sum_{i=0}^{n+1} \partial_{n+1,i} s_{n,i}
      - \sum_{i=1}^{n+1} \partial_{n+1,i} s_{n,i-1} - \sum_{k=1}^{n} (-1)^k
      \sum_{i=k+1}^{n+1} \partial_{n+1,i} s_{n,i-k-1} \\
    &= (-1)^{n+1} \partial_{n+1,0} s_{n,n+1} + \sum_{k=1}^n (-1)^k
      \sum_{i=0}^{n-k+1} s_{n-1,i+k-1} \partial_{n,i} + (n+2) \\
    &\quad - (n+1) - \sum_{k=1}^n (-1)^k \sum_{i=k+1}^{n+1}
      s_{n-1,i-k-1} \partial_{n,i-1} \\
    &= 1 + (-1)^{n+1} \partial_{n+1,0} s_{n,n+1}
      + \sum_{k=1}^n (-1)^k \sum_{i=0}^{n-k+1} s_{n-1,i+k-1} \partial_{n,i} \\
    &\quad - \sum_{k=1}^n (-1)^k \sum_{i=k+1}^{n+1} s_{n-1,i-k-1} \partial_{n,i-1} .
  \end{align*}
  The result follows.
\end{proof}

The Karoubi operator of a duplicial module $M_\bullet$ is determined by the
restriction of the Karoubi operator to its normalization $N_\bullet(M)$:
\begin{equation*}
  \kappa_n = \sum_{k=0}^{n-1} (-1)^k \sum_{0\le i_1<\cdots<i_k<n}
  g_{n-k,i_1\ldots,i_k} \bigl( \kappa_{n-k} p_{n-k} \bigr) f_{n,i_1\ldots i_k} .
\end{equation*}
This follows from the following lemma.
\begin{lemma}
  \label{sk}
  \begin{equation*}
    \kappa_{n+1} s_{n,i} =
    \begin{cases}
      0 , & i=0 \\
      - s_{n,i-1} \kappa_n , & 1\le i\le n \\
      \bigl( s_{n,n+1} - s_{n,n} \bigr) \kappa_n , & i=n+1
    \end{cases}
  \end{equation*}
\end{lemma}
\begin{proof}
  If $i=0$, we have
  \begin{align*}
    (-1)^n \kappa_{n+1} s_{n,0} &= \partial_{n+2,0} s_{n+1,n+2} s_{n,0} - s_{n,n+1}
                             \partial_{n+1,0} s_{n,0} \\
                           &= \partial_{n+2,0} s_{n+1,0} s_{n,n+1} - s_{n,n+1}
                             \partial_{n+1,0} s_{n,0} \\
                           &= s_{n,n+1} - s_{n,n+1} = 0 .
  \end{align*}
  If $1\le i\le n$, we have
  \begin{align*}
    (-1)^n \kappa_{n+1} s_{n,i} &= \partial_{n+2,0} s_{n+1,n+2} s_{n,i} - s_{n,n+1}
                             \partial_{n+1,0} s_{n,i} \\
                           &= \partial_{n+2,0} s_{n+1,i} s_{n,n+1} - s_{n,n+1}
                             s_{n-1,i-1} \partial_{n,0} \\
                           &= s_{n,i-1} \partial_{n+1,0} s_{n,n+1} - s_{n,i-1}
                             s_{n-1,n} \partial_{n,0} \\
                           &= (-1)^{n-1} s_{n,i-1} \kappa_n .
  \end{align*}
  Finally, we have
  \begin{align*}
    (-1)^n \kappa_{n+1} s_{n,n+1} &= \partial_{n+2,0} s_{n+1,n+2} s_{n,n+1} - s_{n,n+1}
                             \partial_{n+1,0} s_{n,n+1} \\
                           &= \partial_{n+2,0} s_{n+1,n+1} s_{n,n+1} - s_{n,n+1}
                             \partial_{n+1,0} s_{n,n+1} \\
                           &= \bigl( s_{n,n} - s_{n,n+1} \bigr)
                             \partial_{n+1,0} s_{n,n+1} \\
                           &= (-1)^{n-1} \bigl( s_{n,n+1} - s_{n,n}
                             \bigr) \kappa_n \\
                             &\quad + (-1)^n \bigl( s_{n,n+1} -
                             s_{n,n} \bigr) s_{n-1,n} \partial_{n,0} \\
                           &= (-1)^{n-1} \bigl( s_{n,n+1} - s_{n,n}
                             \bigr) \kappa_n ,
  \end{align*}
  since $s_{n,n+1} s_{n-1,n} = s_{n,n} s_{n-1,n}$.
\end{proof}

\section{The Dwyer--Kan operator}
\label{section:DwyerKan}

Let $M_\bullet$ be a duplicial module. The Dwyer--Kan operator on $M_n$ is
the operator
\begin{equation*}
  \pi_n = \delta_{n+1}\kappa_{n+1}^n\sigma_n .
\end{equation*}
Like the Karoubi operator, it may be expressed in terms of $b$ and
$d$.
\begin{proposition}
  \label{pi}
  We have
  \begin{equation*}
    \pi_n = \kappa_n^n - b_{n+1}\kappa_n^nd_n = (1-b_{n+1}d_n)^{n+1}
    (1-d_{n-1}b_n)^n ,
  \end{equation*}
  and hence $b_n\pi_n=\pi_{n-1}b_n$ and $d_n\pi_n=\pi_{n+1}d_n$.
\end{proposition}

The proof of Proposition \ref{pi} uses Lemma \ref{sk}, together with
the following lemma, which follows from Lemma \ref{sk} on applying the
anti-involution $(-){}^\circ$.
\begin{lemma}
  \label{pk}
  \begin{equation*}
    \partial_{n,i} \kappa_n =
    \begin{cases}
      \kappa_{n-1} \bigl( \partial_{n,0} - \partial_{n,1} \bigr) , & i=0 , \\
      - \kappa_{n-1} \partial_{n,i+1} , & 1\le i<n , \\
      0 , & i=n .
    \end{cases}
  \end{equation*}
\end{lemma}

\begin{proof}[Proof of Proposition \ref{pi}]
  It follows from Lemma~\ref{sk} that
  \begin{equation*}
    \kappa_n^ms_{n,n+1} = \sum_{i=0}^m (-1)^m s_{n,n-m+1} ,
  \end{equation*}
  and hence that
  \begin{align*}
    \partial_{n+1} \kappa_n^ns_{n,n+1} &= (-1)^{n+1} \partial_{n+1,0} ( d_n - s_{n,0} )
    \kappa_n^n \\
    &= (-1)^{n+1} \partial_{n+1,0} \kappa_n^n d_n - (-1)^{n+1} \kappa_n^n .
  \end{align*}
  Likewise, it follows from Lemma~\ref{pk} that
  \begin{align*}
    \partial_{n+1,0} \kappa_n^n d_n &= \kappa_n^n \bigl( b_{n+1} - (-1)^{n+1} \partial_{n+1,n+1}
                          \bigr) d_n \\
                        &= b_{n+1} \kappa_n^n d_n - (-1)^{n+1} \kappa_n^n
                          \partial_{n+1,n+1} d_n .
  \end{align*}
  But $\kappa_n^n \partial_{n+1,n+1} d_n$ vanishes:
  \begin{equation*}
    \kappa_n^n \partial_{n+1,n+1} d_n = \sum_{i=0}^{n-1} (-1)^i \kappa_n^n \partial_{n+1,n+1}
    s_{n,i} = \sum_{i=0}^{n-1} (-1)^i \kappa_n^n s_{n-1,i} \partial_{n,n} ,
  \end{equation*}
  and $\kappa_n^n s_{n-1,i}=0$ by Lemma~\ref{sk}.
\end{proof}

The operator $\pi$ is chain homotopic to the identity for either of the
differentials $b$ or $d$. For the differential $b$, the contracting
homotopy is the Connes operator
\begin{equation*}
  B_n = \sum_{i=0}^n d_n\kappa_n^i : M_n \to M_{n+1} .
\end{equation*}
We have
\begin{align*}
  b_{n+1}B_n + B_{n-1}b_n &= b_{n+1}d_n(1+\cdots+\kappa_n^n) +
  (1+\cdots+\kappa_n^{n-1})d_{n-1}b_n \\
  &= (b_{n+1}d_n+d_{n-1}b_n)(1+\cdots+\kappa_n^{n-1}) + b_{n+1}d_n\kappa_n^n \\
  &= (1-\kappa_n)(1+\cdots+\kappa_n^{n-1}) + b_{n+1}\kappa_n^nd_n = 1 - \pi_n \intertext{and}
  B_{n+1}B_n &= d_{n+1}(1+\cdots+\kappa_{n+1}^{n+1}) d_n(1+\cdots+\kappa_n^n) \\
  &= d_{n+1}d_n(1+\cdots+\kappa_{n+1}^{n+1})(1+\cdots+\kappa_n^n) = 0 .
\end{align*}
Similarly, for the differential $d$, the contracting homotopy is the
operator
\begin{equation*}
  D_n = \sum_{i=0}^{n-1} b_n\kappa_n^i : M_n \to M_{n-1} .
\end{equation*}
This is the operator $\mathbf{i}$ of Ginzburg and Schedler
\cites{GS1,GS2}: they introduce the complex $(M_\bullet\[v\],d+vD)$, where
$v$ is a formal variable of degree $2$, as an alternative to the
complex $(M_\bullet\[u\],b+uB)$ defining negative cyclic homology, where
$u$ is a formal variable of degree $-2$,

More generally, if $\ell\ge1$, the operator
\begin{equation*}
  B(\ell)_n = \sum_{i=1}^{\ell(n+1)-1} d_n\kappa_n^i
\end{equation*}
satisfies the homotopy relation
\begin{equation*}
  b_{n+1}B(\ell)_n + B(\ell)_{n-1}b_n = 1 - \pi^\ell_n ,
\end{equation*}
and the side relation $B(\ell)_{n+1}B(\ell)_n=0$.

We now come to the main result of this article.
\begin{theorem}
  \label{main}
  $\pi_n= p_n \circ \T_n = \T_n \circ p_n$
\end{theorem}
\begin{proof}
  Although identity is stated for duplicial modules, it suffices to
  prove it for paracyclic modules, since the the functor from $\Lambda_+$ to
  $\Lambda_\infty$ is faithful. This simplifies the proof, since we may make use
  of the inverses $t_n^{-1}$.

  For $0\le i<j\le n$, define
  \begin{equation*}
    \Pi_{i,j} =   t_n^i \delta_{n+1} t_{n+1}^{j-i} \sigma_n t_n^{-j-1} .
  \end{equation*}
  We have
  \begin{align*}
    \Pi_{i,j} &= (-1)^n t_n^{i-n-1} \partial_{n+1,0} t_{n+1}^{j-i+n+2}
              s_{n,n+1} t_n^{-j-1} \\
            &= (-1)^{i+1} \partial_{n+1,n-i+1} t_{n+1}^{j+1} s_{n,n+1} t_n^{-j-1} \\
            &= (-1)^{j-i} \partial_{n+1,n-i+1} s_{n,n-j} = (-1)^{j-i} s_{n-1,n-j}
              \partial_{n,n-i} \\
            &= (-1)^{j-i} \bigl( - s_{n-1,i} \partial_{n,i+1} \bigr)
              \bigl( - s_{n-1,i+1} \partial_{n,i+2} \bigr) \ldots
              \bigl( - s_{n-1,j-1} \partial_{n,j} \bigr) .
  \end{align*}
  The Dold--Puppe idempotent may be written using the operators
  $\Pi_{i,j}$:
  \begin{equation*}
    p_n = 1 + \sum_{k=1}^{n-1} \sum_{0\le i_1<j_1<i_2<j_2< \ldots<i_k<j_k\le n} \Pi_{i_1,j_1} \ldots
    \Pi_{i_k,j_k} .
  \end{equation*}
  From the definition of $\pi_n$, we see that
  \begin{align*}
    \pi_n &= \delta_{n+1}(\delta_{n+2}\sigma_{n+1}+\sigma_n\delta_{n+1})^n\sigma_n \\
        &= \T_n + \sum_{k=1}^{n-1} \sum_{0\le i_1<j_1<i_2<j_2< \ldots<i_k<j_k\le n}
          \Pi_{i_1,j_1}\ldots\Pi_{i_k,j_k} \T_n .
  \end{align*}
  The proof is completed by showing that if $s<p$,
  \begin{equation}
    \label{PiPi}
    \Pi_{p,q} \Pi_{r,s} = \Pi_{r,s} \Pi_{p,q} .
  \end{equation}
  This allows us to reverse the order of the factors in
  $\Pi_{n-j_1,n-i_1} \ldots \Pi_{n-j_k,n-i_k}$:
  \begin{equation*}
    \Pi_{n-j_1,n-i_1} \ldots \Pi_{n-j_k,n-i_k} = \Pi_{n-j_k,n-i_k} \ldots \Pi_{n-j_1,n-i_1} .
  \end{equation*}
  To prove \eqref{PiPi}, we must show that
  \begin{equation*}
    s_{n-1,p} \partial_{n,q} s_{n-1,r} \partial_{n,s}
    = s_{n-1,r} \partial_{n,s} s_{n-1,p} \partial_{n,q} .
  \end{equation*}
  This is a consequence of the simplicial relations:
  \begin{align*}
    s_{n-1,p} \partial_{n,q} s_{n-1,r} \partial_{n,s} 
    &= s_{n-1,p} s_{n-2,r} \partial_{n-1,q-1} \partial_{n,s} \\
    &= s_{n-1,r} s_{n-2,p-1} \partial_{n-1,s} \partial_{n,q}
    = s_{n-1,r}  \partial_{n,s} s_{n-1,p} \partial_{n,q} .
    \qedhere
  \end{align*}
\end{proof}

\begin{corollary}
  \label{para}
  A duplicial module $M_\bullet$ is paracyclic if and only if $\kappa_n$ is
  invertible on $N_n(M)$ for all $n\ge0$.
\end{corollary}
\begin{proof}
  Apply the operator $\T_n$ to the Dold--Kan decomposition. Since
  $\T_n$ commutes with the codegeneracy maps, we see that
  \begin{align*}
    \T_n &= \sum_{k=0}^n \sum_{0\le i_1<\cdots<i_k<n} g_{n-k,i_1\ldots i_k}
           \bigl( \T_{n-k}p_{n-k} \bigr) f_{n,i_1\ldots i_k} \\
         &= \sum_{k=0}^n \sum_{0\le i_1<\cdots<i_k<n} g_{n-k,i_1\ldots i_k}
           \bigl( \pi_{n-k}p_{n-k} \bigr) f_{n,i_1\ldots i_k} .
  \end{align*}
  Thus $\T_n$ is invertible on $M_n$ if and only if $\pi_{n-k}$ is
  invertible on $N_{n-k}(M)$ for $0\le k\le n$. But $\pi_n$ is invertible on
  $N_n(M)$ if and only if $\kappa_n$ is, by the formulas
  \begin{equation*}
    \kappa_n^{-1} = \pi_n^{-1} (1+b_{n+1}d_n)^n (1+d_{n-1}b_n)^{n-1}
  \end{equation*}
  and $\pi_n^{-1} = \kappa_n^{-n-1} (1+d_{n-1}b_n)$.
\end{proof}

A \textbf{parachain complex} is a duchain complex $(V_\bullet,b,d)$ such
that the Karoubi operator $\kappa_n=1-b_{n+1}d_n-d_{n-1}b_n$ on $V_n$ is
invertible for all $n\ge0$, or equivalently, such that the operators
$1-b_{n+1}d_n$ and $1-d_{n-1}b_n$ on $V_n$ are invertible for
$n\ge0$. This terminology is due to Getzler--Jones \cite{GJ}, although
this subcategory of duchain complexes occurs implicitly in \cite{DK}.

Applying Corollary \ref{para}, we obtain the following theorem, also
due to Dwyer and Kan.
\begin{theorem}
  The normalization functor $N_\bullet$ induces an equivalence of categories
  between paracyclic modules $M_\bullet$ in $\CA$ and parachain complexes in
  $\CA$.
\end{theorem}

Finally, in light of Theorem~\ref{main}, we recover
\cite{DK}*{Corollary 6.6}, as the case $\ell=1$ of the following theorem.
\begin{theorem}
  The normalization functor $N_\bullet$ induces an equivalence of categories
  between cyclic modules $M_\bullet$ in $\CA$ and parachain complexes in
  $\CA$ such that $\pi_n=b_{n+1}(1-d_nb_{n+1}-b_{n+2}d_{n+1})^nd_n$ has
  order $\ell$ for all $n\ge0$.
\end{theorem}

Equivalently, the parachain complexes obtained by normalizing cyclic
modules are those for which
\begin{equation*}
  b_{n+1}B(\ell)_n + B(\ell)_{n-1}b_n = 0 , \quad n\ge0 .
\end{equation*}

\section{The cyclic structure on a simplicial module}

There is a full and faithful functor from chain complexes
$(V_\bullet,b)$ to parachain complexes, which identifies chain complexes as
those duchain complexes whose extra differential $d$ vanishes. (In
particular, these parachain complexes have $\kappa_n=\pi_n=1$.) We may think
of this functor as a natural embedding of the category of simplicial
modules in the category of cyclic modules (with $\ell=1$). The additional
codegeneracy $\sigma_n$ of this cyclic structure acts by the special case
of \eqref{s} in which the second term on the right-hand side is
omitted:
\begin{equation*}
  \sigma_n g_{n-k,i_1\dots i_k} p_{n-k}
  = \sum_{\ell=0}^k \sum_{i_\ell<i<i_{\ell+1}} (-1)^{i+k-\ell} \, g_{n-k,i_1\ldots i_\ell
    ii_{\ell+1} \ldots i_k} p_k .
\end{equation*}
We illustrate this structure by giving formulas for $t_n$ for small
values of $n$: for $n=1$, we have
\begin{equation*}
  t_1(v+s_{0,0}v_0) = -v+s_{0,0} ( b_1v + v_0 ) ,
\end{equation*}
where $v\in N_1(M)$ and $v_0\in N_0(M)$, and for $n=2$, we have
\begin{multline*}
  t_2(v+s_{1,0}v_0+s_{1,1}v_1+s_{1,1}s_{0,0}v_{01}) \\
  = v+s_{1,0} (-bv-v_0+v_1) + s_{1,1}(bv+v_0) +
  s_{1,1}s_{0,0}(bv_1+v_{01}) ,
\end{multline*}
where $v\in N_2(M)$, $v_0,v_1\in N_1(M)$, and $v_{01}\in N_0(M)$.

In this way, we obtain a new formula for the idempotent $p_n$ of Dold
and Puppe for a simplicial module $M_n$: it equals the Dwyer--Kan
operator $\pi_n$ associated to the canonical cyclic structure on
$M_\bullet$. This puts us in an interesting situation: we now have two
different chain homotopies demonstrating the quasi-isomorphism of the
complexes $(M_\bullet,b)$ and $(N_\bullet(M),b)$ for simplicial modules. For each
$n\ge0$, we have
\begin{equation*}
  b_{n+1}\phi_n + \phi_{n-1}b_n = b_{n+1}B_n + B_{n-1}b_n = 1 - p_n .
\end{equation*}
\begin{proposition}
  The two contracting homotopies $\phi_\bullet$ and $B_\bullet$ are themselves
  homotopic, by the formula
  \begin{equation*}
    \phi_n - B_n = b_{n+2} \psi_n + \psi_{n-1}b_n ,
  \end{equation*}
  where $\psi_n=B_{n+1}\phi_n$.
\end{proposition}
\begin{proof}
  We have
  \begin{align*}
    b_{n+2} \psi_n - \psi_{n-1}b_n &= b_{n+2} \bigl( B_{n+1} \phi_n \bigr) -
                               \bigl( B_n \phi_{n-1} \bigr) b_n \\
    &= \bigl( b_{n+2} B_{n+1} + B_n b_{n+1} \bigr) \phi_n
      - B_n \bigl( b_{n+1} \phi_n + \phi_{n-1} b_n \bigr) \\
    &= \bigl( \phi_n-p_{n+1}\phi_n \bigr) - \bigl( B_n - B_np_n \bigr)
    = \phi_n - B_n .
  \end{align*}
  Here, $p_{n+1}\phi_n$ vanishes by the formula $p_{n+1}s_{n,i}=0$,
  $0\le i\le n$, and $B_np_n$ vanishes for the cyclic structures under
  consideration, because $d_np_n=0$.
\end{proof}

\subsection*{Acknowledgements}

This article would not have been possible without the assistance of
the National Park Service, and especially of the Lowney Creek
Campground in the Pictured Rocks National Lakeshore. The author thanks
Uppsala University and the Knut and Alice Wallenberg Foundation for
their support during the completion of the article.

\end{document}